\documentclass[12pt, reqno]{amsart}
\usepackage{paper_style}

\usepackage[csfont, paper]{std_math}
\usepackage{local}
\usepackage{changing_Macros}

\usepackage{xurl}

\makeatletter
\@addtoreset{enumi}{section}
\makeatother

\hypersetup{pdfauthor={Segev Gonen Cohen}
	,pdftitle=Discontinuous actions on manifolds
	,urlcolor=blue
	,citecolor=red
    ,linkcolor=blue
	,colorlinks=true}

\begin{document}

\title{Discontinuous actions of rational special linear groups on manifolds}
\author{Segev Gonen Cohen}
\address{Department of Mathematics, ETH Zürich, Zürich, Switzerland}
\email{segev.gonencohen@math.ethz.ch}

\thanks{The author acknowledges the support of FNS (Fonds National Suisse, no. 10.003.145).}


\date{\today.}
\begin{abstract}
    For every $n\geq 1$, we construct continuum-many pairwise non-conjugate actions of $\SL(n+1,\bbQ)$ by homeomorphisms on the $n$-sphere. None of these actions extend to a continuous action of $\SL(n+1,\bbR)$. Our construction refines a blow-up procedure of Bowden and Mann, allowing the stabiliser action inside each blown-up disc to be prescribed.  We motivate studying these actions and record some related observations.
\end{abstract}
\maketitle

\tableofcontents

\section{Introduction}

\noindent The purpose of this paper is to show that discontinuous actions by homeomorphisms of $\SL_{n+1}(\bbQ)$ on manifolds are abundant. In all that follows, we will refer to a homomorphism from the $\bbQ$-points of a $\bbQ$-group to a topological group as \emph{continuous} if it is continuous with respect to the topology induced from $\bbR$.

\begin{theo}
\label{theo:uncountably_many}
    For every $n \geq 1$ there are $2^{\aleph_0}$-many pairwise non-conjugate discontinuous homomorphisms $\rho:\SL_{n+1}(\bbQ) \to \Homeo(S^n)$.
\end{theo}

\noindent As motivation for studying these actions, we begin by recalling a special case of the classical Theorem of Borel and Tits. Recall that a connected $\bbQ$-group $\bbG$ is said to be \emph{almost simple} if it has no proper closed normal subgroups of strictly positive dimension, and \emph{absolutely almost simple} if $\bbG(\bbC)$ is almost simple. Recall in addition that a $\bbQ$-group $\bbG$ is \emph{$\bbQ$-isotropic} if it has a $\bbQ$-split torus of strictly positive dimension. A key example to keep in mind for both properties is $\SL_{n+1}$, for $n \geq 1$.

\begin{theo}[{\cite{BorelTits1973}}, special case]
\label{theo:Borel_Tits}
    Let $\bbG$ be an absolutely almost simple, simply connected, $\bbQ$-isotropic $\bbQ$-group. Let $\bbH$ be an absolutely simple $\bbQ$-group. Then every abstract group homomorphism
    $$\phi:\bbG(\bbQ) \to \bbH(\bbR)$$
    with Zariski-dense image is the restriction to $\bbG(\bbQ)$ of an $\bbR$-defined isogeny
    \begin{equation*}
        \tilde{\phi}:\bbG(\bbR) \to \bbH(\bbR).
    \end{equation*}
    In particular, $\phi$ is continuous.
\end{theo}

\noindent Recently, there has been an extended body of research in proving `non-linear' analogues of results known to hold for linear groups. Here by `non-linear' we specifically consider groups of diffeomorphisms of compact manifolds. See for example Hurtado's work on the Burnside problem in diffeomorphism groups \cite{hurtado2020burnside} and the progress on the Zimmer programme \cite{brown_zimmers_2020, BrownFisherHurtado2020, BrownFisherHurtado2021preprint, Cantat2019, brown_entropy_2019, fisher_recent_2019} as important results in this direction. Specifically for actions of the group of $\bbQ$-points of algebraic $\bbQ$-groups, Witte Morris and Zimmer \cite[Theorem~1.1]{MorrisZimmer2012} prove that every smooth, volume-preserving action of $\SL_{n+1}(\bbQ)$ on a compact manifold of dimension at most $n$ is trivial. They also establish that such actions have to have finite image for more general algebraic $\bbQ$-groups under some natural hypotheses, \cite[Theorem~1.3]{MorrisZimmer2012}. See also \cite[Remark~1.2]{MorrisZimmer2012} for related topological obstructions to the existence of actions by homeomorphisms.

It seems natural therefore to ask to what extent the theorem of Borel and Tits transfers to the non-linear setting. Below $M$ will always denote a compact connected manifold with no boundary. $\Diff^1(M)$ will denote the group of $C^1$-diffeomorphisms of it equipped with the $C^1$-topology, and $\Homeo(M)$ the group of homeomorphisms of $M$ equipped with the compact-open topology. We will use $_+$ to indicate that the maps we are considering are all orientation-preserving. In light of Theorem~\ref{theo:Borel_Tits}, the following conjecture is natural.

\begin{conjecture}
\label{conj:extensions_of_Actions}
    Let $\bbG$ be an absolutely almost simple, simply connected $\bbQ$-group, and $M$ a compact manifold. Let $\alpha: \bbG(\bbQ) \to \Diff^1(M)$ be a homomorphism. Then it extends up to finite index to a continuous homomorphism $\tilde{\alpha}:\bbG(\bbR) \to\Diff^1(M)$. In particular, up to finite-index, $\alpha$ is continuous.
\end{conjecture}

\noindent Note that here, in contrast to Theorem~\ref{theo:Borel_Tits}, we have not assumed that $\bbG$ is $\bbQ$-isotropic. In fact, in the case that $\rank{\bbQ}{\bbG} = 0$, we have some partial results, for manifolds of a low enough dimension. We give just one example, deferring to \cite{cohen2026actionslatticespadicsarithmetic, Gonen_Cohen_semisimple} for a description of known results.

\begin{theo}[{\cite{Gonen_Cohen_semisimple}}] \label{theo:Q_groups}
    Let $n\geq 2$, and let $\bbG$ be a simple $\bbQ$-group with $\rank{\bbQ}{\bbG} = 0$ such that $\bbG(\bbR) \cong \SL_{n+1}(\bbR) \times \SL_{n+1}(\bbC)$. Let $M$ be a compact manifold of dimension $\dim M < 2n$. Let $\alpha:\bbG(\bbQ) \to \Diff^3(M)$ be a homomorphism. Then it extends up to finite index to a continuous homomorphism $\tilde{\alpha}:\SL_{n+1}(\bbR) \to \Diff^2(M)$.
\end{theo}

\noindent The construction of a group $\bbG$ satisfying the assumptions of Theorem~\ref{theo:Q_groups} can be done by restriction of scalars, see for example \cite{Margulis1991}. Notice that in this case, $\bbG(\bbQ)$ is different to $\SL_{n+1}(\bbQ)$; and $\bbG$ is not absolutely almost-simple. We record in Proposition~\ref{prop:C^r_regularity} the best known statement for the group $\SL_{n+1}(\bbQ)$ in dimension $n$, when further regularity is assumed, which follows from work of Brown, Rodriguez-Hertz, and Wang \cite{brown_boundary_2024}.

We now turn our attention to actions by homeomorphisms. Recall that if $X_1,X_2$ are topological spaces endowed with group actions $\rho_i:G \to \Homeo(X_i)$, a \emph{(topological) semi-conjugacy} between $\rho_1$ and $\rho_2$ is a continuous surjective map $\tau:X_1\to X_2$ such that for every $g\in G$ we have that $\tau\circ\rho_1(g) = \rho_2(g) \circ \tau$. In this case we say that $\rho_1$ is \emph{left semi-conjugate} to $\rho_2$, and $\rho_2$ is \emph{right semi-conjugate} to $\rho_1$.
We say that $\tau$ is a \emph{(topological) conjugacy} if it is a homeomorphism.

\begin{prop}[{\cite[Proposition 4.1]{BowdenMann2022}}]
\label{prop:Bowden_Mann}
    Let $k\geq 1$ and consider the $k$-sphere $S^k$. Let $G$ be a countable group, and $\rho:G \to \Diff^1_+(S^k)$ an action with a dense orbit $Y$. Then there exists some $\rho':G \to \Homeo_+(S^k)$ and a continuous semi-conjugacy $\tau:S^k \to S^k$ between $\rho'$ and $\rho$ such that 
    \begin{enumerate}[label = {(\roman*)}]
        \item For every $y\in Y$ the fiber $D_y :=\tau^{-1}(y)$ is a closed topological disc, and
        \item $\tau$ is one-to-one on $S^k\setminus\tau^{-1} (Y)$.
    \end{enumerate}
\end{prop}

\begin{rem}
    For actions on $S^1$, this is related to the classical Denjoy blow-up (see for example \cite[Section 4.4.3]{HasselblattKatok2003}), which only requires the initial action $\rho$ to be by homeomorphisms. For abelian groups, the Denjoy construction can even be smoothed to a $C^1$-action. We do not expect this to be true in general. It is also remarked in \cite{BowdenMann2022} that this proposition should hold for other manifolds $M$ (other than just the $k$-spheres), but showing this would require some care.
\end{rem}

\noindent We now restrict our discussion to the group $\bbG =\SL_{n+1}$, for ease of exposition. Recall that for $n \geq 1$ the group $\SL_{n+1}(\bbR)$ admits an analytic action $\sigma:\SL_{n+1}(\bbR) \to \Diff^\omega_+(S^n)$ via $\sigma(g) v := \frac{gv}{\|gv\|}$, henceforth referred to as the \emph{projective action} of $\SL_{n+1}(\bbR)$. This action is transitive and continuous, so any dense subgroup of $\SL_{n+1}(\bbR)$ admits a dense orbit on $S^n$. There is also the \emph{dual projective action} given by $\sigma^\vee (g) :=\sigma((g^{-1})^T)$.

\begin{cor}
\label{cor:discontinuous_action}
    For every $n \geq 1$ there exists a discontinuous homomorphism
    $$\rho':\SL_{n+1}(\bbQ)\to \Homeo_+(S^{n}).$$
\end{cor}

\begin{proof}
    Let $G = \SL_{n+1}(\bbQ)$. Apply the Bowden-Mann construction (Proposition~\ref{prop:Bowden_Mann}) to the projective action $\sigma|_{G}$ and any dense orbit $Y = \sigma(G)\cdot x$ to obtain the action $\rho':G \to \Homeo_+(S^n)$ and the semi-conjugacy $\tau:S^n\to S^n$. It is shown in their proof that in fact the interiors of the $D_y$ are pairwise disjoint. We denote the interior of a set $A$ by $\operatorname{int}(A)$.
    
    To show that this $\rho$' is discontinuous, notice that the stabiliser of $x$ (denoted $G_x$) is not open in the topology induced from $\SL_{n+1}(\bbR)$, so we can find a sequence $g_j \in G\setminus G_x$ such that $g_j \to \Id$. Take any point $z \in \operatorname{int}(D_x)$, then $\rho'(g_j)z \in \operatorname{int}(D_{g_jx})$ and since these are all disjoint, $\rho'(g_j)z \nrightarrow z$.
\end{proof}

\noindent The use of blow-ups to obtain exotic actions of higher rank semisimple groups and their lattices goes back to Katok and Lewis \cite{KatokLewis1996}, and was extended by Benveniste and Fisher \cite{Benveniste1996, BenvenisteFisher2005}. See also \cite{Fisher2019GroupsActingOnManifoldsZimmer} and the references therein for more. 

For $n\geq 2$, the dimension of the sphere in Theorem~\ref{theo:uncountably_many} is optimal. Indeed, Bridson and Vogtmann \cite[Corollary~1.4]{BridsonVogtmann2011} show that every action by homeomorphisms of $\SL_{n+1}(\bbZ)$ on a sphere $S^k$ for $k < n$ is trivial.  Since $\SL_{n+1}(\bbZ)$ normally generates $\SL_{n+1}(\bbQ)$, the same holds for $\SL_{n+1}(\bbQ)$. Similar results are known for many other arithmetic groups with torsion, see \cite{BridsonGrunewaldVogtmann2014}.

\begin{rem}
\label{rem:p_Adic_lattice}
    Let $G$ be any semisimple Lie group that acts transitively on the sphere; these have been classified by Onishchik \cite{Onishchik1967}. Then the same argument as in Corollary~\ref{cor:discontinuous_action} works for any dense countable subgroup $\Gamma \leq G$. In particular, let $\Gamma \leq \SU_{n+1}^k(\bbR) \times \SL_{n+1}(\bbQ_p)$ be a lattice (for some $k\geq 1$) that projects densely onto some $\SU_{n+1}(\bbR)$-factor, then it must admit actions by homeomorphisms on $S^{2n+1}$ that do not extend to a continuous action of $\SU_{n+1}(\bbR)$. Any lattice in $\SL_{n+1}(\bbQ_p)$ for $n\geq 2$ is commensurable to one of this form (by Margulis' Arithmeticity Theorem \cite[Theorem~IX.1.11]{Margulis1991}). Compare this to \cite[Question 2]{cohen2026actionslatticespadicsarithmetic}, which asks if this is possible for $C^1$-actions of lattices in $\SL_{n+1}(\bbQ_p)$.
\end{rem}

\noindent To obtain the uncountably many non-conjugate actions (in fact, non-semi-conjugate actions) in Theorem~\ref{theo:uncountably_many}, we adapt the Bowden-Mann construction to allow for more flexibility. Specifically, we are able to prescribe the action of the stabiliser of a point in the blown-up disc; this is done in Proposition~\ref{prop:Bowden_mann_specified}. However, the actions in Theorem~\ref{theo:uncountably_many} are all still (left) semi-conjugate to the projective one. This motivates the following.

\begin{que}
    Let $n\geq 2$. Does there exist a non-trivial homomorphism $\rho:\SL_{n+1}(\bbQ) \to \Homeo_+(S^n)$ that is not (left) semi-conjugate to the projective action $\sigma$ or the dual projective action $\sigma^\vee$?
\end{que}

\noindent Analogous questions can be asked for the groups discussed in Remark~\ref{rem:p_Adic_lattice}, including lattices in $p$-adic groups. It seems plausible that a negative solution to the Hilbert-Smith conjecture could be used to construct such examples. We note that on 06 October 2026, which was after this paper was written and circulated, OpenAI claimed a proof of the Hilbert-Smith conjecture \cite{OpenAI2026HilbertSmith}. At the time of writing this result has not (to the author's knowledge) been verified by a human or Lean-formalised.

\subsection*{Outline of the paper} In Section~\ref{sect:blow_up}, we extend the Bowden-Mann construction by specifying the action of $G_x$ on the disc $D_x$, see Proposition~\ref{prop:Bowden_mann_specified}. As a warm-up, in Section~\ref{sect:Sl_2_on_S2} we use this construction to give a discontinuous action of $\SL_2(\bbQ)$ on $S^2$, starting from any discontinuous action of $\bbQ$ on the 2-dimensional disc, these are plentiful. In Section~\ref{sect:higher_rank} we use this construction to prove Theorem~\ref{theo:uncountably_many}.

In Section~\ref{sect:circle_Actions} we explicate another way to obtain a discontinuous $\SL_2(\bbQ)$-action on $S^1$, recalling some facts about groups acting on circles. Finally, in Section~\ref{sect:final_observations} we collect some additional observations about Conjecture~\ref{conj:extensions_of_Actions}, including some topological obstructions to actions in low dimension.

\subsection*{Acknowledgements} I would like to primarily thank Andreas Wieser, whose interest in this question and our conversations led to this work. I would like to thank Marc Burger for suggesting these questions, and Menny Akka and Manfred Einsiedler for helpful comments and discussions.

\subsection*{Disclosure on AI use}

GPT-5.6-Sol was used for literature search and initial proof ideation, in particular the author was unfamiliar with the Bowden-Mann construction previously. The key arguments building on it are the author's own. The text in this paper is human-written, and the author takes sole responsibility for any inaccuracies. Choices of content, exposition, and explanation are entirely the author's. GPT-6-Astra was used to proof-read.

\section{Blow-ups with specified interior maps}
\label{sect:blow_up}

In this section we prove the key proposition, which we will apply to prove Theorem~\ref{theo:uncountably_many} later on. Bowden and Mann prove their blow-up construction \cite[Proposition~4.1]{BowdenMann2022} in the case that the stabiliser of any point in the dense orbit is trivial, and remark that their construction extends to non-trivial stabilisers using the natural action of the stabiliser on the tangent space. We require a slightly stronger version, allowing us to prescribe the action of the stabiliser on the disc $D_x$ explicitly, subject to a compatibility requirement. 

Before we state it we fix some notation. For a real vector space $V$, we let $\bbP^+(V):= (V\setminus \{0\})/\bbR_{>0}$ denote the space of directed rays in $V$ based at 0, we denote for a non-zero vector $v\in V$ the corresponding ray by $[v] \in \bbP^+(V)$. For two real vector spaces $V,W$ and a linear isomorphism $T:V\to W$ we denote by $[T]:\bbP^+(V) \to \bbP^+(W)$ the corresponding projectivisation. $\bbD^k$ denotes the open unit disc in $\bbR^k$, and $\bbB^k$ its closure. We consistently identify $\bbP^+(V) \cong\partial \bbB^k$ when this causes no confusion.

\begin{prop}
\label{prop:Bowden_mann_specified}
    Let $k\geq 1$. Let $G$ be a countable group and let $\rho:G \to \Diff^1_+(S^k)$ be a homomorphism with a dense orbit $Y:= \rho(G)\cdot x$ for some $x\in S^k$. Let $H = G_x$ be the stabiliser of $x$ in $G$. Let $\alpha:H\to \Homeo_+(\bbB^k)$ be a homomorphism that satisfies 
    \begin{equation*}
        \alpha(h)[v] = [D_x\rho(h)v] \text{ for all } h\in H \text{ and } [v] \in \bbP^+(T_xS^k),
    \end{equation*}
    where we identified $\partial \bbB^k \cong \bbP^+(T_xS^k)$.

    Then there exists some $\hat{\rho}_\alpha:G \to \Homeo_+(S^k)$
    and a continuous semi-conjugacy $\tau:S^k \to S^k$ between $\hat{\rho}_\alpha$ and $\rho$ such that
    \begin{enumerate}[label = {(\roman*)}]
        \item For every $y\in Y$ the fiber $D_y :=\tau^{-1}(y)$ is a closed topological disc,
        \item $\tau$ is one-to-one on $S^k\setminus\tau^{-1} (Y)$, and
        \item After identifying $D_x$ with $\bbB^k$ one has 
        \begin{equation*}
            \hat{\rho}_\alpha(h)|_{D_x} = \alpha(h)
        \end{equation*}
        for all $h \in H$.
    \end{enumerate}
\end{prop}

\begin{proof}
    We complete the proof in multiple steps. Where convenient we will drop the $\rho$ and denote the diffeomorphism $\rho(g)$ simply by $g$.
    
   \smallskip
    \noindent\textit{Step 1: The Bowden-Mann construction.}
    \smallskip
    
    \noindent By Proposition~\ref{prop:Bowden_Mann} there is a manifold $M$ homeomorphic to $S^k$, a homomorphism $\rho':G \to \Homeo_+(M)$, and a surjective, continuous map $\tau:M \to S^k$ such that for every $y\in Y = G\cdot x$ the preimage $D_y:=\tau^{-1}(y)$ is a closed disc, it is one-to-one on $M\setminus \tau^{-1}(Y)$, and such that $\tau \circ \rho'(g) =\rho(g)\circ \tau$  for every $g \in G$. That is, we just need to prove property (iii). We choose to denote the blown-up $S^k$ by $M$ for the sake of readability, despite it being homeomorphic to $S^k$.

    For any compatible metric on $M$, from the proof of \cite[Proposition 4.1]{BowdenMann2022} it follows that for any sequence of pairwise distinct elements $y_n \in Y$ we have that $\operatorname{diam}(D_{y_n}) \to 0 $.

    Letting $K := M \setminus \bigcup_{y\in Y}\operatorname{int}(D_y)$, the restriction of $\rho'(g)$ to $K$ sends $\partial D_y \to \partial D_{g y}$. Under the canonical identification $\partial D_y \cong \bbP^+(T_yS^k)$, the boundary map is given by
    \begin{equation*}
        \rho'(g)|_{\partial D_y}: \partial D_y \mapsto \partial D_{gy}:[v] \mapsto [D_y\rho(g)v].
    \end{equation*}

    \smallskip
    \noindent\textit{Step 2: the stabiliser-valued cocycle.} 
    \smallskip
    
    \noindent For each $y \in Y$ choose a representative $s_y \in G$ such that $s_y x = y$, we choose for $x$ the representative $s_x = e$.  Define for $g \in G$ and $y\in Y$ the element $c(g,y) := s_{g y}^{-1}g s_y$. It is easy to check that this is a cocycle (that is, $c(g_1g_2,y) = c(g_1,g_2y)c(g_2,y)$), and since $c(g,y)x = s_{g y}^{-1}g y = x$ in fact $c(g,y) \in H$ for all $g\in G$ and $y\in Y$.

    \smallskip
    \noindent\textit{Step 3: Defining $\hat{\rho}_\alpha$.}
    \smallskip
    
    \noindent Choose an orientation preserving homeomorphism $T_x:\bbB^k \to D_x$ that agrees on the boundary with the identification $\partial\bbB^k  \cong \partial D_x \cong \bbP^+(T_xS^k)$, this is always possible by the Alexander trick. For every $y\in Y$ set 
    \begin{equation*}
        T_y:= \rho'(s_y)\circ T_x:\bbB^k \to D_y.
    \end{equation*}
     Consider some $g\in G$. On $K$, set $\hat{\rho}_\alpha(g)$ to agree with $\rho'(g)$. On $D_y$, set
    \begin{equation*}
        \hat{\rho}_\alpha(g):= T_{g y}\alpha(c(g,y))T_y^{-1}.
    \end{equation*}
    By the chain rule and the fact that $s_{g y}c(g,y)s_y^{-1} = g$ it follows that the two definitions agree on $\partial D_y$ and so $ \hat{\rho}_\alpha(g)$ is well-defined.

    \smallskip
    \noindent\textit{Step 4: $\hat{\rho}_\alpha$ is a group homomorphism.}
    \smallskip
    
     \noindent The group law holds on $K$ so we just need to verify it on the $D_y$. Indeed, on $D_y$, 
    \begin{align*}
        \left(\hat{\rho}_\alpha(g_1)\hat{\rho}_\alpha(g_2)\right)|_{D_y} & = \hat{\rho}_\alpha(g_1)|_{D_{g_2y}}\cdot \hat{\rho}_\alpha(g_2)|_{D_y} \\
        & = T_{g_1g_2 y}\alpha(c(g_1,g_2y))T_{g_2y}^{-1}T_{g_2 y}\alpha(c(g_2,y))T_y^{-1} \\
        & = T_{g_1g_2 y}\alpha(c(g_1,g_2y)c(g_2,y))T_y^{-1} \\
        & = T_{g_1g_2 y}\alpha(c(g_1g_2,y))T_y^{-1} \\
        & = \hat{\rho}_\alpha(g_1g_2)|_{D_y}
    \end{align*}
    as required. Since $c(h,x) = h$ for all $h \in H$, this gives the required property (iii).

    \smallskip
    \noindent\textit{Step 5: For every $g\in G$, $\hat{\rho}_\alpha(g)$ is a homeomorphism.}
    \smallskip
    
     \noindent The only difficult case is to consider a sequence $z_n \in \operatorname{int}(D_{y_n})$ where $z_n \to z$ and the $y_n$ are pairwise distinct, we want to show that $\hat{\rho}_\alpha(g)z_n \to \hat{\rho}_\alpha(g)z$. 
     
     Recall that necessarily in this case $\operatorname{diam}(D_{y_n})\to 0$, so choose some $p_n \in \partial D_{y_n}$, and then $d(z_n,p_n) \leq \operatorname{diam}(D_{y_n}) \to 0$. So in particular we also have that $p_n \to z$, and so $z\in K$. Since the new action equals $\rho'$ on the boundary, we have that $\hat{\rho}_\alpha(g)p_n = \rho'(g)p_n \to \rho'(g)z = \hat{\rho}_\alpha(g)z$.
     
     The points $gy_n$ are also pairwise distinct, and so we again have that $\operatorname{diam}(D_{g y_n})\to 0$. Since $\hat{\rho}_\alpha(g)z_n$ and $\hat{\rho}_\alpha(g)p_n$ lie in $D_{g y_n}$ we have that $d(\hat{\rho}_\alpha(g)z_n,\hat{\rho}_\alpha(g)p_n) \to 0$, and so $\hat{\rho}_\alpha(g)z_n \to \hat{\rho}_\alpha(g)z$ as required.

    By repeating the argument for $g^{-1}$ we see that in fact $\hat{\rho}_\alpha(g)^{-1} = \hat{\rho}_\alpha(g^{-1}) $ is also continuous, and so $\hat{\rho}_\alpha(g)$ is a homeomorphism. Its restriction between disc interiors is orientation preserving and hence since the sphere is connected, $\hat{\rho}_\alpha(g)$ is orientation preserving. Finally, fixing an orientation-preserving homeomorphism $M \cong S^k$ gives the result. 
\end{proof}

\section{A discontinuous action of \texorpdfstring{$\SL_2(\bbQ)$ on $S^2$}{}}
\label{sect:Sl_2_on_S2}

\noindent We apply Proposition~\ref{prop:Bowden_mann_specified} to obtain many discontinuous actions of $G = \SL_2(\bbQ)$ on the $2$-sphere $S^2$. This is not necessary for Theorem~\ref{theo:uncountably_many} and may be skipped, but we find it illustrative. We need a $C^1$-action of $G$ on $S^2$ with a dense orbit, and a prescribed action of the stabiliser of a point on the disc; we start with the latter. Let $B = \{x\in \bbR^2:\|x\|<1/4\} \subset \bbB^2$. By \cite[Proposition 1.3]{GrunderSerraille2026} we have a homomorphism
\begin{equation*}
    \Phi:\bbQ \to \Diff^\infty_c(B)
\end{equation*}
that is discontinuous in the $C^0$-topology, and in fact many such examples exist. Since for every $t\in \bbQ$ the map $\Phi(t)$ has compact support in $B$, it is the identity on some neighbourhood of $\partial B$. So we can extend it by the identity to get a smooth diffeomorphism of $\bbB^2$, denoting this extension also by $\Phi:\bbQ\to \Homeo_+(\bbB^2)$.

Morally speaking, to obtain an action of $\SL_2(\bbQ)$ on $S^2$ we would like to consider the action of $\SL_2(\bbQ) $ on $\bbR^2$ and take the one-point compactification. However, this results in an action that is not $C^1$. To get around this, we consider the embedding 
\begin{equation*}
    \imath:\SL_2(\bbR) \to \SL_3(\bbR) : g \mapsto \begin{pmatrix}
        g & 0 \\ 0 & 1
    \end{pmatrix}.
\end{equation*}
and consider the projective action $\sigma:\SL_3(\bbR) \to \Diff^\omega_+(S^2)$. Consider the point $x = \frac{1}{\sqrt{2}}(1,0,1) = \frac{1}{\sqrt{2}}(e_1,1)$ where $e_1 = (1,0)$. Then for any $g \in \SL_2(\bbQ)$, 
\begin{equation*}
    \sigma(\imath(g))x = \frac{(ge_1,1)}{\sqrt{\|ge_1\|^2+1}}
\end{equation*}
Since $G\cdot e_1 = \bbQ^2 \setminus \{0\}$ and the fact that $v \mapsto \frac{(v,1)}{\sqrt{\|v\|^2+1}}$ is a homeomorphism from $\bbR^2$ to the open upper hemisphere, we see that $G \cdot x$ is dense in the closed upper hemisphere. To obtain density in the closed lower hemisphere we take an extension of $\SL_2(\bbQ)$ by the group with two elements, as follows.

Let $R = \diag(1,-1,-1) \in \SO(3)$, and consider
$$G' = \{\imath(g)R^\varepsilon : g\in \SL_2(\bbQ),\varepsilon \in \{0,1\}\} \leq \SL_3(\bbQ),$$
this is isomorphic as a group to $G \rtimes_\theta \bbZ/2\bbZ$ where $\theta(g) = JgJ^{-1}$ and $J = \diag(1,-1)$. As before, $GR \cdot x$ is dense in the closed lower hemisphere, and so $G'\cdot x$ is dense in $S^2$. Denote the projective action $\sigma$ restricted to $G'$ by $\rho$, so we have an action by $\rho:G'\to \Diff^\omega_+(S^2)$. 

The stabiliser at $x$ is $G'_x = \{\imath(u(t)):t\in \bbQ\} \cong \bbQ$, where $u(t) = (\begin{smallmatrix}
    1 & t \\ 0 & 1
\end{smallmatrix})$. In fact for any $t\in \bbR$ the matrix $\imath(u(t))$ fixes $x$, and so gives some map
\begin{equation*}
    \delta_t:= \bbP^+(D_x\rho(\imath(u(t))))\in \Homeo(\bbP^+(T_xS^2)).
\end{equation*}
Extending this to $\bbR$ we have a flow $(\delta_t)_{t\in \bbR}$ on $\bbP^+(T_xS^2)$, which we can identify concretely. A basis of $T_xS^2$ is given by $f_1 = (1/2,0,-1/2)$, $f_2 = (0,1,0)$, and a direct calculation shows that
$$D_x\rho(\imath(u(t)))(af_1+bf_2) = (a+tb)f_1+bf_2.$$
That is, with respect to this basis $\delta_t$ is the unipotent flow. We finally need to modify $\Phi$ so that it matches this on $\partial \bbB^2$, which we identify with $\bbP^+(T_xS^2)$.

Choose a continuous function $\psi:[0,1]\to [0,1]$ such that $\psi(r) = 0$ for $r\leq 1/3$ and $\psi(r) = 1$ for $r \geq 2/3$. For any $t\in \bbR$ define a map $C_t \in \Homeo(\bbB^2)$ via 
$$C_t(0) = 0, \quad C_t(rv) = r \delta_{\psi(r)t}v \ \text{for } r >0, \ v \in \partial \bbB^2.$$
In particular, $C_t|_B = \Id_B$. It follows that 
\begin{equation*}
    C_s\Phi(t) = \Phi(t)C_s
\end{equation*}
for all $s\in \bbR$ and $t\in \bbQ$. Let $\alpha(u(t)) := C_t\Phi(t)$, this gives a homomorphism $\alpha:G'_x \to \Homeo(\bbB^2)$ such that 
\begin{enumerate}
    \item $\alpha(\imath(u(t)))|_{\partial \bbB^2} = \delta_t$, and
    \item $\alpha(\imath(u(t)))|_{B} = \Phi(t)$.
\end{enumerate}
We are thus in a position to apply Proposition~\ref{prop:Bowden_mann_specified}; the resulting action $\hat{\rho}_\alpha:G' \to \Homeo_+(S^2)$ is discontinuous, since $\Phi$ was, completing the construction.

\section{Discontinuous actions of \texorpdfstring{$\SL_{n+1}(\bbQ)$ on $S^n$}{}}
\label{sect:higher_rank}

\noindent In this section we prove Theorem~\ref{theo:uncountably_many} by using Proposition~\ref{prop:Bowden_mann_specified}. Let $G = \SL_{n+1}(\bbQ)$ and recall the projective action $\sigma:G\to \Diff^\omega_+(S^n)$. Let $x = e_1 = (1,0,\ldots,0)^T$, clearly its orbit is dense, and its stabiliser is
\begin{equation*}
    H = G_x = \left\{h = \begin{pmatrix}
        a & v \\ 0 & B
    \end{pmatrix}: a \in \bbQ_{>0},\ B \in \GL_n(\bbQ),\ a\det B = 1,\ v \in \bbQ^n\right\}.
\end{equation*}
We can identify $T_xS^n $ with $e_1^\perp \cong \bbR^n$, then by differentiation we see that for $h \in H$, 
\begin{equation*}
    D_x\sigma(h) = a^{-1}B.
\end{equation*}
Indeed, let $N(y) := \frac{y}{\|y\|}$, then $\sigma(h) = N \circ h$ and (viewing $S^n$ as an embedded submanifold of $\bbR^{n+1}$), $D_{x}\sigma(h) = D_{ax}N\circ h$. Then notice that for non-zero $y\in \bbR^{n+1}$, $D_yN(w) = \frac{1}{\|y\|}\left(1-\frac{yy^T}{\|y\|^2}\right)w$. So for $v \in e_1^\perp$,
\begin{equation*}
    D_x\sigma(h)v = a^{-1}(1-e_1e_1^T)h v = a^{-1}Bv
\end{equation*}
as claimed.

To use Proposition~\ref{prop:Bowden_mann_specified}, we need actions of $H$ on $\bbB^n$ matching the boundary conditions. Let $\psi:\bbQ_{>0} \to \bbR_{>0}$ be any homomorphism, we can compose this with the projection $H \to \bbQ_{>0}:h \mapsto a$ to obtain a homomorphism $\chi:H \to \bbR_{>0}$. Now set for $h\in H$,
\begin{equation*}
    A_\psi(h):= \chi(h)D_x\sigma(h) = \psi(a)a^{-1}B \in \GL_n^+(\bbR)
\end{equation*}
(where the plus denotes the connected component of the identity in $\GL_n(\bbR)$, the matrices with a positive determinant). For any $A \in \GL_n(\bbR)$ we can consider the diffeomorphism $\Theta_A \in \Diff^\infty(\bbB^n)$ defined by
\begin{equation*}
    \Theta_A(x) := \frac{Ax}{\sqrt{1-\|x\|^2+\|Ax\|^2}},
\end{equation*}
this gives a homomorphism $\Theta:\GL_n(\bbR) \to \Diff^\infty(\bbB^n)$. Finally, set $\alpha(h) = \alpha_\psi(h) := \Theta_{A_\psi(h)}$, giving a homomorphism $\alpha:H \to \Diff^\infty_+(\bbB^n)$. For $v\in \partial \bbB^n$ notice that $\alpha(h)[v] = [D_x\sigma(h)v]$, as is required to apply Proposition~\ref{prop:Bowden_mann_specified} to $\sigma, x$, and $\alpha$, to obtain the action $\hat{\rho}_\alpha:G \to \Homeo_+(S^n)$. It is easy to see that this action is discontinuous when $\psi$ is, without using the argument of Corollary~\ref{cor:discontinuous_action}. We write the full details for the readers' benefit.

Indeed, any homomorphism $\psi :\bbQ_{>0}\to \bbR_{>0}$ is of the form $\psi(r) = \prod c_p^{\nu_p(r)}$ where the product is taken over all primes $p$, $c_p = \psi(p) >0$ are real numbers, and $\nu_p:\bbQ^\times\to \bbZ$ denotes the $p$-adic valuation. Let $\lambda_p = \ln c_p/\ln p$. 

Suppose now that $\psi$ is discontinuous, equivalently there is no $c \in \bbR$ such that $\psi(r) = r^c$ for all $r\in \bbQ_{>0}$. In particular we can choose distinct primes $p\neq q$ such that $\lambda_p \neq \lambda_q$. Let $\theta = \ln(q)/\ln(p)$ and choose continued fraction convergents $a_k/ b_k\to \theta$, in particular $|a_k/b_k-\theta|<1/b_k^2$. Let $s_k = p^{a_k}/q^{b_k}$, and notice that $\ln s_k \to 0$ and hence $s_k \to 1$ as $k \to \infty$. However, 
\begin{align*}
    \ln \psi(s_k) & = a_k \ln \psi(p) - b_k \ln \psi(q) \\
    & = a_k \lambda_p \ln p - b_k \lambda_q\ln q\\
    & = \lambda_p(a_k\ln p-b_k \ln q) + (\lambda_p-\lambda_q)b_k \ln q
\end{align*}
and so $|\ln \psi(s_k)| \to \infty$ as $k \to \infty$. Now set
$$h_k = \begin{pmatrix}
    s_k^n & 0 \\ 0 & s_k^{-1}\Id_n
\end{pmatrix},$$ 
notice that $h_k \to \Id$ in $\SL_{n+1}(\bbR)$. Write $A_\psi(h_k) = \psi(s_k)^{n}s_k^{-n-1}\Id =: C_k\Id$, and notice that $C_k \to 0$ or $\infty$ as $k \to \infty$, and for any unit vector $v \in \bbR^n$
\begin{equation*}
    \left\|\Theta_{A_{\psi}(h_k)}\left(\frac{v}{\sqrt{2}}\right)-\frac{v}{\sqrt{2}}\right\| = \left|\frac{C_k}{\sqrt{1+C_k^2}}-\frac{1}{\sqrt{2}}\right| > \frac{1}{4}
\end{equation*}
eventually, whether $C_k \to 0$ or $\infty$. So this action is indeed discontinuous.

All that is left to do is to determine if for  discontinuous homomorphisms $\psi_1 \neq \psi_2$ from $\bbQ_{>0} $ to $\bbR_{>0}$, the corresponding actions $\rho_1:= \hat{\rho}_{\alpha_{\psi_1}}$ and $\rho_2:= \hat{\rho}_{\alpha_{\psi_2}}$ are conjugate. Recall that each comes with a semi-conjugacy $\tau_i$ to the standard projective action $\sigma$.

Suppose first that $\beta$ is any other semi-conjugacy between $\rho_1$ and $\sigma$. Notice that 
\begin{equation*}
    U = \left\{\begin{pmatrix}1 & v\\0 & \Id\end{pmatrix}:v\in \bbQ^{n}\right\}
\end{equation*}
acts trivially by $\rho_1$ on $D_x$, and so $\beta(D_x)$ lies in the fixed-point set of $\sigma(U)$, which is just $\{\pm x\}$. Since $D_x$ is connected and $\beta$ is continuous, it is constant on $D_x$. The same argument holds for every $y \in Y$, since this is a dense set this completely determines $\beta$ up to the antipodal map $\pm \Id$, and so $\beta = \pm \tau_1$.

Now suppose that $\gamma$ is a semi-conjugacy from $\rho_1$ to $\rho_2$. In particular, $\tau_2 \circ \gamma$ is a semi-conjugacy from $\rho_1$ to $\sigma$, and so by the above $\tau_1 = \pm \tau_2 \circ \gamma$. Suppose for the sake of argument that $\tau_1 = \tau_2 \circ \gamma$, the other case is similar. In particular, $\gamma$ maps $D_x^1 := \tau_1^{-1}(x)$ onto $D_{x}^2 := \tau_2^{-1}(x)$ in an equivariant way, and thus we must have that
\begin{equation*}
    \ker(a \mapsto \psi_1(a)^na^{-n-1}) \subset  \ker(a \mapsto \psi_2(a)^na^{-n-1}),
\end{equation*}
and we can clearly pick continuum-many characters that pairwise do not satisfy this; thus the proof of Theorem~\ref{theo:uncountably_many} is complete. \qed

\section{Circular orders and actions on circles}
\label{sect:circle_Actions}

\noindent In this section, we recall some facts about group actions on circles, and in particular give another construction of a discontinuous action of $\SL_2(\bbQ)$ on the circle. First, recall the following definition.

\begin{defn}
    A set $S$ with at least 4 elements is \emph{circularly ordered} if for every $p\in S$ there is a choice of total ordering $<_p$ on $S\setminus \{p\}$, such that for every set of four pairwise distinct elements $p,q,x,y\in S$ with $x<_p y$, we have that $<_p$ and $<_q$ are the same on $\{x,y\}$, unless $x<_pq<_py$ in which case $y<_qp<_qx$.
\end{defn}
Given a circular order on $S$ we can define the function $c:S^3 \to \{-1,0,+1\}$ by setting $c(x,y,z) = 1$ if $y <_x z$, $-1$ if the triple is negatively ordered, and 0 if the points are not pairwise distinct. We define the \emph{order interval} to be
\begin{equation*}
    (x,z)_c := \{y\in S:c(x,y,z)=1\}
\end{equation*}
A non-empty subset of $S$ is said to be \emph{convex} if it is the intersection of a set of order intervals. We say that a circular order is \emph{dense} if $(x,z)_c \neq \emptyset$ whenever $x\neq z$. The following is a simple verification.

\begin{lem}
\label{lem:circular_and_linear}
    Let $S$ be a circularly ordered set, and for every $x\in S$ let $L_x$ be a linearly ordered set. Then $T = \bigsqcup_{x\in S} L_x$ has a circular order such that:
    \begin{enumerate}
        \item Each $L_x$ is convex;
        \item For $a_i \in L_{x_i}$, $i=1,2,3$ with the $x_i$ distinct, the order is prescribed by the order of the $x_i$'s;
        \item Within each $L_x$ the circular order comes from the linear order.
    \end{enumerate}
    If in addition $S$ and the $L_x$ are all countable and dense, then $T$ is countable and dense.
\end{lem}
The connection between circular orders and the circle $S^1$ is classical.

\begin{lem}[e.g. {\cite[pg. 2281]{Calegari2006UniversalCircles}}]
\label{lem:dense_circular_order}
    Let $S$ be a countable set with a dense circular order. Then $S$ is circular order isomorphic to $\bbQ/\bbZ \subset S^1$. In addition, every circular order-automorphism of $S$ extends uniquely to an element of $\operatorname{Homeo}_+(S^1)$.
\end{lem}

\begin{defn}
    A countable group $G$ is \emph{(left-)circularly orderable} if it admits a circular order which is preserved by the left action on itself.
\end{defn}
\noindent We also recall the following, even though we technically will not use it.

\begin{theo}[{\cite[Theorem 2.2.14]{Calegari2004Circular}}]
    A countable group $G$ is circularly orderable if and only if it admits a faithful action on $S^1$ by orientation-preserving homeomorphisms, which furthermore can be taken so that (at least) one point has a trivial stabiliser.
\end{theo}
\noindent To obtain a discontinuous action of $G = \SL_2(\bbQ)$ by homeomorphisms on $S^1$ we will construct a circularly ordered $G$-set. Let $G$ act linearly on $\bbR^2$, and consider the vector $e_1 = (1,0)^T$ which has orbit $G\cdot e_1 = \bbQ^2\setminus \{0\}$. Letting $U = G_{e_1}=\{u(t) = (\begin{smallmatrix}
    1 & t \\ 0&1
\end{smallmatrix}):t\in \bbQ\}$ we have that the orbit map $G/U \to \bbQ^2\setminus \{0\}:gU \to ge_1$ is a $G$-equivariant bijection.

\begin{lem}
    There is a circular order on $G/U$ that is dense and $G$-invariant.
\end{lem}

\begin{proof}
    For any $v \in \bbQ^2\setminus \{0\}$ consider the unit-vector representative $[v] = v/\|v\|$ and the ray $L_{[v]} = \bbR_{>0} v \cap \bbQ^2\setminus \{0\}$. Then $S = \{[v]:v\in \bbQ^2\setminus \{0\}\}$ inherits a dense circular order from the circle, and each $L_{[v]}$ comes with the standard (dense) total order. Applying Lemma~\ref{lem:circular_and_linear} we obtain a dense circular order on $\bbQ^2 \setminus \{0\}$ which is clearly $G$-invariant, so we can pull it back to $G/U$.
\end{proof}
\noindent Consider now $\bbQ^2$ with the lexicographic order $\ell$, that is $(a,b) <_\ell (a',b')$ if either $a<a'$ or $a=a'$ and $b<b'$. $U$ acts on $\bbQ^2$ by translation via $u(t)(a,b) = (a+t,b)$. This preserves the order, and notice that $u(t)0 \notin (-e_2,e_2)_\ell$ for all $t \neq 0$. Define the suspension 
\begin{equation*}
    X = G \times \bbQ^2/U
\end{equation*}
where the quotient is by $(gu,v)\sim (g, uv)$, and we denote this equivalence class by $[g,v]$. There is a natural projection $\pi:X\to G/U$.
\begin{lem}
    Every fiber of $\pi$ has a well-defined linear order that is isomorphic to $(\bbQ^2,\ell)$.
\end{lem}

\begin{proof}
    We have for any $g \in G$ a bijection $\bbQ^2 \to \pi^{-1}(gU):v \mapsto [g,v]$. Replacing $g$ by $gu$ changes the fiber coordinate by $v \mapsto uv$, but this is an order-preserving map; so the order is independent of the choice of $g$.
\end{proof}
\noindent Now apply Lemma~\ref{lem:circular_and_linear} to the dense circular order on $G/U$, and the linear order on each fiber over $G/U$, to obtain a dense circular order on $X$. By Lemma~\ref{lem:dense_circular_order} it is circular order isomorphic to $\bbQ/\bbZ$, say via some map $\phi:X \to \bbQ/\bbZ$. Since it is circular order-preserving, order intervals in $X$ map to the intersections of open intervals in $S^1$ with $\bbQ/\bbZ$.

$G$ acts on $X$ by circular order-automorphisms. Indeed, note that for any $g,h\in G$ and $u\in U$ we can write $h[g,v]=[hg,v] = [hgu^{-1}u,v] = [hgu^{-1},uv]$. This clearly preserves the circular order on $G/U$, and the order on each fiber of $\pi$. So by Lemma~\ref{lem:dense_circular_order} the action of each $g$ on $\phi(X)$ extends uniquely to some $\rho(g) \in \operatorname{Homeo}_+(S^1)$. Uniqueness of extension implies that $\rho(gh) = \rho(g)\rho(h)$, so we have an action $\rho:\SL_2(\bbQ) \to \Homeo_+(S^1)$.

We finally show that it is discontinuous. For this, let $x_- =\phi([e,-e_2])$, $x_0 = \phi([e,0])$, and $x_+ = \phi([e,e_2])$, where $e\in G$ denotes the identity element. Notice that $(x_-,x_+)$ contains $x_0$. However, for any $t\in \mathbb{Q}$ we have that $\rho(u(t))x_0 = \phi([u(t),0]) = \phi([1,u(t)0]) \notin (x_-,x_+)$ for $t\neq 0$, as was noted previously. But since $u(1/n) \to \Id$ in $\SL_2(\bbR)$, the fact that $\rho(u(1/n))x_0$ never enters the fixed neighbourhood $(x_-,x_+)$ of $x_0$ proves discontinuity.

\begin{rem}
     This action is not conjugate to any of the ones considered in Theorem~\ref{theo:uncountably_many}. Indeed, the $G$ action on $X$ is free, so no non-identity element fixes an open interval in $S^1$. In the construction in Theorem~\ref{theo:uncountably_many}, non-identity elements of $U$ fix the discs $D_x$ pointwise.
\end{rem}

\section{Final observations}
\label{sect:final_observations}

\noindent We record in this section some final observations surrounding Conjecture~\ref{conj:extensions_of_Actions} and actions by homeomorphisms. We claim no originality of ideas or methods here; indeed we expect them to be known to the experts, but have not seen them explicated anywhere. 

The first is that for actions of $\SL_{n+1}(\bbQ)$ (for $n \geq 2$) on manifolds of dimension $n$, the regularity matters. We record the following proposition, which states that for regularity $C^r$ for $r>1$, only the projective action and its dual are possible. By $\Diff^{k+\beta}(M)$ for $k \in \bbN$ and $0 \leq \beta < 1$ we mean the set of diffeomorphisms that are $C^k$, and such that their $k^{th}$-derivatives are $\beta$-Hölder. This is a consequence of \cite[Theorem 1.1]{brown_boundary_2024}.

\begin{prop}
\label{prop:C^r_regularity}
    Let $r>1$ and $n \geq 2$, and let $\rho:\SL_{n+1}(\bbQ) \to \Diff^r(M)$ be a homomorphism, where $M$ is a compact manifold of dimension $n$. Then either $\rho$ is trivial, or $M = S^n$ or $\bbR\bbP^n$ and there is some $f \in \Diff^r(M)$ that conjugates $\rho$ to the standard projective action or its dual. In the latter case, in particular $\rho$ extends to a continuous homomorphism from $\SL_{n+1}(\bbR)$.
\end{prop}

\begin{proof}
    Let $G = \SL_{n+1}(\bbQ)$ and $\Gamma = \SL_{n+1}(\bbZ)$, and recall that both groups are perfect (in fact, they are generated by elementary matrices $E_{ij}(t) = \Id + t\delta_{ij}$). Suppose first that $\rho(\Gamma)$ is finite, then for some $m \geq 1$ in fact $\rho(E_{12}(m)) = \Id$. Now for any $t \in \bbQ^\times$ let $a_t = \diag(t/m,1,m/t,\Id_{n-2})$ and notice that $a_tE_{12}(m)a_t^{-1} = E_{12}(t)$, and so $E_{12}(t) \in \ker(\rho)$. The same holds for every $E_{ij}(t)$, showing that $\rho$ is trivial. So suppose from now on that $\rho(\Gamma)$ is infinite.

    By applying \cite[Theorem~2.3]{brown_boundary_2024} it follows that $M \cong S^n$ or $\bbR\bbP^n$, and there is a $C^r$-conjugacy $f\in \Diff^r(M)$ to the standard projective action of $\Gamma$ (or its dual, we denote either action by $\sigma$). That is, for every $\gamma \in \Gamma$ , $\sigma(\gamma) = f\rho(\gamma)f^{-1}$. Define $\rho_f:G \to \Diff^r(M) : g\mapsto f \rho(g) f^{-1}$, we need to show that this agrees with $\sigma$ on all of $G$.

    Recall that $G$ commensurates $\Gamma$, that is for every $g \in G$ the group $\Gamma_g := \Gamma \cap g^{-1}\Gamma g$ has finite index in $\Gamma$. For any $\gamma \in \Gamma_g$ since both $\gamma$ and $g\gamma g^{-1}$ are in $\Gamma$, we have that
    \begin{equation*}
        \rho_f(g)\sigma(\gamma)\rho_f(g^{-1}) = \rho_f(g)\rho_f(\gamma)\rho_f(g^{-1}) = \rho_f(g \gamma g^{-1}) = \sigma (g \gamma g^{-1}) = \sigma(g)\sigma(\gamma)\sigma(g^{-1})
    \end{equation*}
    and so $c(g) := \sigma(g)^{-1}\rho_f(g)$ centralises $\sigma(\Gamma_g)$. By Lemma~\ref{lem:centraliser_lemma} it must be either the identity or the antipodal map, but since $G$ is perfect we have that $c(g) = 1$ for all $g$ and the result follows. 
\end{proof}

We show the following for the $n$-sphere, it is easy to adapt it to the case of $\bbR\bbP^n$.

\begin{lem}
    \label{lem:centraliser_lemma}
    Let $\Gamma = \SL_{n+1}(\bbZ)$ and $\sigma:\Gamma \to \Diff^\omega_+(S^n)$ be the projective action. Let $\Lambda \leq \Gamma$ be a finite index subgroup. Let $c\in \Homeo(S^n)$ centralise $\sigma(\Lambda)$. Then $c = \Id$ or $c = -\Id$.
\end{lem}

\begin{proof}
    Consider any primitive vector $v \in \bbZ^{n+1}$, let $w = v/\|v\|$, and choose a matrix $\gamma \in \Gamma$ such that $\sigma(\gamma) e_1 = w$. Since $\Lambda$ has finite index, for each $j = 2,\ldots,n+1$ there are $m_j \geq 1$ such that
    \begin{equation*}
        \gamma E_{1j}(m_j)\gamma^{-1} \in \Lambda.
    \end{equation*}
    It is easy to see that the common fixed points on $S^n$ of these maps are $\left\{\pm w\right\}$, and so if $c$ centralises $\Lambda$ it must satisfy $c\left(w\right) \in \left\{\pm w\right\}$. Since $\{v/\|v\|:v$ is primitive$\}\subset S^n$ is dense and $S^n$ is connected, the result follows.
\end{proof}

Generally, for $\bbZ/2\bbZ$-homology spheres of dimension $<n$, it is known that actions of $\SL_{n+1}(\bbQ)$ by homeomorphisms have to be trivial, by \cite[Corollary 1.4]{BridsonVogtmann2011}. For general manifolds of dimension $< n$ the best known result is \cite[Theorem 1]{brown_c1_2022}, which states that $C^1$-actions of $\SL_{n+1}(\bbZ)$ have to have finite image. Combining this with the calculation in the proof of Proposition~\ref{prop:C^r_regularity}, which shows that finite-index subgroups of $\SL_{n+1}(\bbZ)$ normally generate $\SL_{n+1}(\bbQ)$, gives triviality for actions of $\SL_{n+1}(\bbQ)$ by $C^1$-diffeomorphisms on any compact manifold of dimension $<n$.

Next, we remark that the Bowden-Mann construction (Proposition~\ref{prop:Bowden_Mann}) and our specification of it (Proposition~\ref{prop:Bowden_mann_specified}) can give discontinuous actions for $\SL_{n+1}(\bbQ)$ by homeomorphisms on $S^n \times N$ for any compact manifold $N$, and on $S^k$ ($k \geq n$) and $\bbB^m$ ($m\geq n+1$) respectively by taking the suspension or cone respectively over the actions on $S^{k-1}$, by recursion. It might be tempting therefore to embed $\bbD^m$ in any compact manifold $M$, and extend the action by the identity outside $\bbD^m$ to obtain a discontinuous action on $M$. However this is not possible, by a classical result of Newman.

\begin{prop}
    Let $G$ be a group with a torsion element $g$ whose normal closure is all of $G$, and $M$ a connected manifold. If a homomorphism $\rho:G \to \Homeo(M)$ is the identity on a non-empty open subset of $M$, then $\rho$ is trivial.
\end{prop}

\noindent This applies to $\SL_2(\bbQ)$ by considering $g =(\begin{smallmatrix}
    0 & -1 \\ 1 & 0
\end{smallmatrix})$, and similarly to any $\SL_n(\bbQ)$ by considering $g$ placed in the top left-hand corner.

\begin{proof}
    Newman's Theorem \cite[Theorem 2]{Newman1931} says that a periodic homeomorphism of a connected manifold without boundary which fixes a non-empty open set pointwise must be the identity. This applies to $\rho(g)$, and so $g \in \ker \rho$ and the result follows.
\end{proof}
\noindent In some cases there are also genuine topological obstructions to constructing any actions. Recall that $G = \SL_{n+1}(\bbQ)$ is not residually finite, in fact such $G$ has no finite quotients, by the argument we gave at the start of the proof of Proposition~\ref{prop:C^r_regularity}. Using this, the fact that mapping class groups of compact surfaces are residually finite (a result of Grossman, \cite{Grossman1974}), and some facts about torsion elements in $\Homeo_0(\Sigma_g)$, we can rule out actions by homeomorphisms on compact hyperbolic surfaces.

\begin{prop}
    Let $G$ be a countable group admitting no finite quotients, and with a torsion element $h \in G$ whose normal closure is all of $G$. Let $\Sigma_g$ be a compact surface of genus $g \geq 2$. Then any homomorphism $\rho:G \to \Homeo(\Sigma_g)$ is trivial.
\end{prop}

\noindent As noted previously, this applies to $G = \SL_{n+1}(\bbQ)$ for all $n \geq 1$.

\begin{proof}
    Consider the induced map to the mapping class group $\overline{\rho}: G \to \operatorname{Mod}(\Sigma_g)$. The latter is residually finite \cite{Grossman1974}, and since $G$ admits no finite quotients we must have that $\overline{\rho}(G) = 1$, that is, $\rho(G) \subset \Homeo_0(\Sigma_g)$, the group of homeomorphisms of $\Sigma_g$ isotopic to the identity. By Lemma~\ref{lem:isotopic_is_torsion_free} this group is torsion-free, so $\rho(h) = 1$. By assumption on $h$, this completes the proof.
\end{proof}

\begin{lem}[e.g. {\cite[Lemma 4.4]{JacoShalen1977}}]
    \label{lem:isotopic_is_torsion_free}
    Let $\Sigma_g$ be a compact surface of genus $g \geq 2$. Then $\Homeo_0(\Sigma_g)$ is torsion-free.
\end{lem}

\printbibliography[title=References]
\end{document}